\documentclass[12pt,a4paper,dvipsnames]{article}

\usepackage[utf8]{inputenc}
\usepackage[backend=biber, giveninits=true, style=ieee, dashed=false, sorting = nty,citestyle=numeric]{biblatex}

\renewbibmacro{in:}{}
\DeclareFieldFormat{pages}{#1}

\AtEveryBibitem{\clearfield{month}
                \clearfield{day}
                \clearfield{doi}}

\usepackage[english]{babel}
\usepackage{amsmath}
\usepackage{amsthm}
\usepackage{amssymb}
\usepackage{geometry}
\usepackage{a4wide}
\usepackage{bm}
\usepackage{microtype}
\usepackage{mathtools}
\usepackage{csquotes}
\usepackage{authblk}

\usepackage{enumitem}
\setlist[enumerate]{itemsep=0mm,parsep=2mm,topsep=5pt,label=\textit{(\alph*)}}
\setlist[itemize]{itemsep=1mm,parsep=2mm,topsep=6pt}

\usepackage{xcolor}
\newcommand\myshade{85}
\colorlet{myurlcolor}{Aquamarine}
\usepackage{hyperref}
\hypersetup{
  linkcolor  = black,
  citecolor  = black,
  urlcolor   = myurlcolor!\myshade!black,
  colorlinks = true,
}
\usepackage[capitalise, noabbrev, nameinlink, nosort]{cleveref}
\crefname{equation}{}{}
\crefname{appsec}{Appendix}{Appendices}

\usepackage{tikz}
\def\normaledge{1.2}
\definecolor{edgeblack}{rgb}{0.25,0.25,0.25}
\definecolor{vertexblack}{rgb}{0.2,0.2,0.2}

\theoremstyle{definition}
\newtheorem{theorem}{Theorem}[section]

\newtheorem{lemma}[theorem]{Lemma}
\newtheorem*{lemma*}{Lemma}
\newtheorem*{conjecture*}{Conjecture}
\newtheorem*{lemma''*}{``Lemma''}
\newtheorem{claim}[theorem]{Claim}
\newtheorem*{claim*}{Claim}
\newtheorem{corollary}[theorem]{Corollary}

\newtheorem{conjecture}[theorem]{Conjecture}

\newcommand{\RR}{\mathbb{R}}

\newcommand{\EE}{\mathbb{E}}
\newcommand{\E}{\mathbb{E}}
\newcommand{\Prob}{\mathbb{P}}
\newcommand{\PP}{\mathbb{P}}
\newcommand{\degpi}{\deg_{G,\overleftarrow{\pi}}}
\newcommand{\degGpi}{\deg_{G_\pi^D,\overleftarrow{\pi}}}

\newcommand{\Npi}{N_{G,\overleftarrow{\pi}}}

\newcommand{\binomialdist}[1]{{\rm Bin}(n,#1)}

\newcommand{\rigmatroidt}{\mathcal{R}_d^t}
\newcommand{\graphicrigid}{\mathcal{M}_d}
\newcommand{\D}{a}
\newcommand{\DD}{{\vec{G}}}

\title{\bf Highly connected orientations from edge-disjoint rigid subgraphs}
\author[1,2]{Dániel Garamvölgyi} 
\author[1,3]{Tibor Jordán}
\author[1]{Csaba Király} 
\author[1,3]{Soma~Villányi} %add affiliations
\affil[1]{\footnotesize HUN-REN-ELTE Egerváry Research Group on Combinatorial Optimization, Pázmány~Péter~sétány~1/C, Budapest, 1117, Hungary}
\affil[2]{\footnotesize HUN-REN Alfréd Rényi Institute of Mathematics, Reáltanoda utca 13-15, Budapest, 1053, Hungary}
\affil[3]{\footnotesize Department of Operations Research, ELTE Eötvös Loránd University, Pázmány~Péter sétány~1/C, Budapest, 1117, Hungary}
\affil[ ]{\footnotesize \textit{E-mail addresses:} {\tt \{daniel.garamvolgyi,tibor.jordan,csaba.kiraly,soma.villanyi\}@ttk.elte.hu}}

\date{}
\begin{document}

\maketitle

\begin{abstract}
    We give an affirmative answer to a long-standing conjecture of Thomassen, stating that every sufficiently highly connected graph has a $k$-vertex-connected orientation. We prove that a connectivity of order $O(k^2)$ suffices. As a key tool, we show that for every pair of positive integers $d$ and $t$, every $(t \cdot h(d))$-connected graph contains $t$ edge-disjoint $d$-rigid (in particular, $d$-connected) spanning subgraphs, where $h(d) = 10d(d+1)$. This also implies a positive answer to the conjecture of Kriesell that every sufficiently highly connected graph $G$ contains a spanning tree $T$ such that $G-E(T)$ is $k$-connected. 
\end{abstract}

\section{Introduction}

It follows from a classical theorem of Nash-Williams that every $2k$-edge-connected graph has a $k$-arc-connected orientation. In 1985, Thomassen asked whether a similar statement is true for vertex-connectivity.

\begin{conjecture}\cite[Conjecture 10]{thomassen_1989}\label{conjecture:thomassen}
For every positive integer $k$ there exists a (smallest) integer $f(k)$ such that every $f(k)$-connected graph has a $k$-connected orientation.
\end{conjecture}

\cref{conjecture:thomassen} has a long history. It is well-known that if $f(k)$ exists, then $f(k) \geq 2k$. Thomassen, together with Jackson, also posed the stronger conjecture that $f(k) = 2k$ (\cite[Conjecture 11]{thomassen_1989}), and later Frank gave the even stronger conjecture that a graph has a $k$-connected orientation if and only if it remains $(2k-2j)$-edge-connected after the deletion of any set of $j<k$ vertices (\cite[Conjecture 7.8]{frank_network}).

The $k = 1$ case of Frank's conjecture is the well-known theorem of Robbins \cite{robbins_1939}; this implies that $f(1) = 2$. The $k=2$ case of \cref{conjecture:thomassen} was proved by the second author \cite{jordan_2005} by showing that $f(2) \leq 18$. Subsequently, Thomassen \cite{thomassen_2015} proved the $k=2$ case of Frank's conjecture, hence establishing that $f(2) = 4$. However, Durand de Gevigney \cite{duranddegevigney_2020} recently disproved Frank's conjecture for $k \geq 3$. He also showed that for such $k$, deciding whether a graph has a $k$-connected orientation is NP-hard.

In this paper we give an affirmative answer to \cref{conjecture:thomassen}, for all $k$, by showing that $f(k) = O(k^2)$.

\begin{theorem}\label{theorem:main2}
    Every $(320 \cdot k^2)$-connected graph has a $k$-connected orientation.
\end{theorem}

The bound on $f(k)$ given by \cref{theorem:main2} is probably far from being tight. In particular, it is still open whether $f(k) = 2k$ holds.

The key new tool in our proof is a packing theorem for highly connected graphs (\cref{theorem:main1} below) that is interesting on its own right. Our original motivation for investigating such packing questions came from the following conjecture of Kriesell from 2003. 

\begin{conjecture}\label{conjecture:Kriesell} (See, e.g., \cite[Problem 444]{mohar.etal_2007})
    For every positive integer $k$ there exists a (smallest) integer $g(k)$ such that every $g(k)$-connected graph $G$ contains a spanning tree $T$ for which $G - E(T)$ is $k$-connected.
\end{conjecture}

As with \cref{conjecture:thomassen}, the edge-connected version of \cref{conjecture:Kriesell} is classical: it follows from a well-known theorem of Nash-Williams \cite{nash-williams_1961} and Tutte \cite{tutte_1961} that every $(2k+2)$-edge-connected graph $G$ contains a spanning tree $T$ such that $G - E(T)$ is $k$-edge-connected.
In particular, we have $g(1) = 4$.
The $k=2$ case of \cref{conjecture:Kriesell} was answered by the second author. In fact, we have the following ``packing theorem'' for $2$-rigid graphs. (Definitions are given in the next section.)

\begin{theorem}\label{theorem:jordan}\cite[Theorem 3.1]{jordan_2005}
    Every $6t$-connected graph contains $t$ edge-disjoint $2$-rigid (and hence $2$-connected) spanning subgraphs. In particular, $g(2) \leq 12$.
\end{theorem}

This bound was subsequently improved to $g(2) \leq 8$ in \cite{szigeti}, where the authors proved an analogous packing result for the union of the $2$-dimensional generic rigidity matroid and the graphic matroid.

The $k=3$ case was settled in a similar fashion by the first three authors. In this case, the underlying matroid was the $\mathcal{C}^1_2$-cofactor matroid, which is conjectured to be the same as the $3$-dimensional generic rigidity matroid.

\begin{theorem}\label{theorem:cofactor}\cite[Theorem 5.11]{garamvolgyi.etal_2024}
    Every $12t$-connected graph contains $t$ edge-disjoint $\mathcal{C}^1_2$-rigid (and hence $3$-connected) spanning subgraphs. In particular, $g(3) \leq 24$.
\end{theorem}

Our second main result is a similar packing result for the $d$-dimensional generic rigidity matroid.

\begin{theorem}\label{theorem:main1}
    Every $\left(t \cdot 10d(d+1)\right)$-connected graph contains $t$ edge-disjoint $d$-rigid (and hence $d$-connected) spanning subgraphs.
\end{theorem}

The existence of a constant $h(d)$ such that every $(t \cdot h(d))$-connected graph contains $t$ edge-disjoint $d$-connected spanning subgraphs was conjectured by the first three authors \cite[Conjecture 5.10]{garamvolgyi.etal_2024}.
The bound in \cref{theorem:main1} is almost certainly not tight. In fact, we believe that every $t\cdot d(d+1)$-connected graph contains $t$ edge disjoint $d$-rigid spanning subgraphs, and hence that $h(d) \leq d(d+1)$. For $t = 1$, this was recently proved by the fourth author.

\begin{theorem}\cite[Theorem 1.1]{soma}\label{theorem:soma}
    Every $d(d+1)$-connected graph is $d$-rigid.
\end{theorem}

We show that for packing $d$-rigid spanning subgraphs, the bound $t\cdot d(d+1)$ would be optimal (\cref{lemma:tightexample}).

We prove \cref{theorem:main1} in \cref{section:packing}, and in \cref{section:Thomassen} we use it to derive \cref{theorem:main2}. In \cref{section:Kriesell} we further investigate the conjecture of Kriesell.
\cref{theorem:main1} implies that Conjecture \ref{conjecture:Kriesell} is true with $g(d) \leq 20d(d+1)$. To improve upon this result, we adapt the proof technique of \cref{theorem:soma} in \cite{soma} to the union of the $d$-dimensional rigidity matroid and the graphic matroid. This leads to the following bound.
% to obtain the following improved bound for $g(d)$.

\begin{theorem}\label{theorem:Kriesellimproved}
    Every $(d^2+3d+5)$-connected graph $G$ contains edge-disjoint spanning subgraphs $G_0$ and $T$ such that $G_0$ is $d$-rigid and $T$ is a tree. In particular, $g(d) \leq d^2 + 3d + 5$. 
\end{theorem}

The bound given by \cref{theorem:Kriesellimproved} is still not optimal. We believe that every $(d(d+1)+2)$-connected graph contains edge-disjoint copies of a $d$-rigid spanning subgraph and a spanning tree. Again, this bound would be tight (\cref{lemma:tightexample}).

An immediate corollary of Theorem \ref{theorem:Kriesellimproved} is that if $G$ is sufficiently highly connected, then
for each pair $s,t\in V(G)$ there exists a path $P$ from $s$ to $t$ in $G$ such that $G-E(P)$ is $k$-connected.
The existence of such paths was verified earlier in \cite{kawarabayashi.etal_2008}, assuming that $G$ is $(1600k^4+k+2)$-connected. 
With Theorem \ref{theorem:Kriesellimproved} the connectivity requirement can be substantially weakened.

\section{Preliminaries}

We start by setting some notation. Throughout the paper we only consider simple graphs, that is, graphs without loops and parallel edges.
%\footnote{We note that the results presented in the Introduction remain true for multigraphs, since the vertex-connectivity of a multigraph is unchanged by 
%deleting loops and parallel copies of an edge.
%passing to the underlying simple graph.
%} 
For a graph $G$, we let $V(G)$ and $E(G)$ denote the vertex and edge sets of $G$, respectively. For a subset $X \subseteq V(G)$ we let $G[X]$ denote the subgraph of $G$ induced by $X$, and we let $i_G(X) = |E(G[X])|$ be the number of edges induced by $X$ in $G$. We use $K(X)$ to denote the complete graph on vertex set $X$, and similarly, $K_n$ denotes the complete graph on $n$ vertices.  Given a vertex $v \in V(G)$, $N_G(v)$ is the set of neighbors of $v$ and $\deg_G(v) = |N_G(v)|$ is the degree of $v$ in $G$. Given a positive integer $k$, we say that a connected graph is \emph{$k$-connected} if it has at least $k+1$ vertices and it remains connected after the removal of any set of fewer than $k$ vertices.

For a directed graph $D$ and a vertex $v \in V(D)$, we use $N_D^-(v)$ and $N_D^+(v)$ to denote the set of in-neighbors and out-neighbors of $v$, respectively. (A vertex $v\in V-X$ is an \emph{in-neighbor} of the vertex set $X$ in the directed graph $D$ if there is an edge $vx$ of $D$ with $x\in X$. Out-neighbors are defined analogously.) We let $\rho_D(v) = |N_D^-(v)|$ and $\delta_D(v) = |N_D^+(v)|$ denote the in-degree and out-degree, respectively, of $v$ in $D$.
A directed graph is \emph{strongly connected} if it contains a directed $u$-$v$ path for every pair of vertices $u$ and $v$, and it is \emph{$k$-connected}, for a positive integer $k$, if it has at least $k+1$ vertices and it remains strongly connected after the removal of any set of fewer than $k$ vertices.

\subsection{Rigidity matroids}

Next, we recall the relevant definitions and facts from rigidity theory. 
%of the $d$-dimensional rigidity matroid. 
For completeness, we start by giving the geometric definition of the generic $d$-dimensional rigidity matroid. In fact, we will only use some combinatorial properties of this matroid, which we collect below.
We assume familiarity with the basic notions of matroid theory; the standard reference is \cite{oxley}. For a more thorough introduction to rigidity theory, see, e.g., \cite{handbook,whiteley_1996}.

Let $G = (V,E)$ be a graph and let $d$ be a positive integer. A \emph{($d$-dimensional) realization} of $G$ is a pair $(G,p)$, where $p : V(G) \rightarrow \RR^d$ maps the vertices of $G$ to $d$-dimensional Euclidean space. A realization is \emph{generic} if its coordinates do not satisfy any nonzero polynomial with rational coefficients.  We identify the space of all $d$-dimensional realizations with $(\RR^d)^V$, and we define the \emph{measurement map} $m_{d,G}: (\RR^d)^V \rightarrow \RR^E$ by \[m_{d,G}(p) = \left(||p(u) - p(v)||^2\right)_{uv \in E}\]
The \emph{rigidity matrix} $R(G,p)$ of a realization $(G,p)$ is the Jacobian of $m_{d,G}$ evaluated at the point $p$. This is a matrix whose rows are indexed by the edges of $G$, and hence the row matroid of $R(G,p)$ can be viewed as a matroid on ground set $E$. It is known that this row matroid is the same for every generic $d$-dimensional realization of $G$. Thus we define the \emph{$d$-dimensional rigidity matroid} of $G$, denoted by $\mathcal{R}_d(G)$, to be the row matroid of $R(G,p)$ for some generic $d$-dimensional realization $(G,p)$. We let $r_d$ denote the rank function of $\mathcal{R}_d$. Using a slight abuse of terminology, we also use $r_d(G)$ to denote the rank of $\mathcal{R}_d(G)$.

We say that a graph $G = (V,E)$ is \emph{$d$-rigid} if $r_d(G) = r_d(K(V))$, and \emph{minimally $d$-rigid} if it is $d$-rigid but $G-e$ is not, for every $e \in E$. Finally, $G$ is \emph{$\mathcal{R}_d$-independent} if $r_d(G) = |E|$. In other words, $G$ is $d$-rigid (resp.\ minimally $d$-rigid, $\mathcal{R}_d$-independent) if and only if $E$ is a spanning set (resp.\ base, independent set) in $\mathcal{R}_d(K(V))$. It is folklore that a graph is $1$-rigid if and only if it is connected. A combinatorial characterization (and an efficient deterministic recognition algorithm) is also available for $2$-rigid graphs, but finding such a characterization for $d$-rigid graphs is a major open question for $d \geq 3$.

As we noted above, we shall only use some well-known combinatorial properties of $\mathcal{R}_d(G)$. These are as follows.

\begin{enumerate}
    \item For $n \geq d$, $r_d(K_n) = dn - \binom{d+1}{2}.$ 
    \item Hence if $G = (V,E)$ is a minimally $d$-rigid graph on at least $d$ vertices, then $|E| = d|V| - \binom{d+1}{2}$, and 
    \begin{equation}\label{eq:sparse}
    i_G(X) \leq d|X| - \binom{d+1}{2}
    \end{equation}
    holds for every subset $X \subseteq V$ of vertices with $|X| \geq d$.
    \item The addition of a vertex of degree $d$ to a graph preserves $\mathcal{R}_d$-independence as well as $d$-rigidity.
    \item If a graph on at least $d+1$ vertices is $d$-rigid, then it is $d$-connected.
    \item If $G_1$ and $G_2$ are $d$-rigid graphs with at least $d$ vertices in common, then $G_1 \cup G_2$ is also $d$-rigid.
\end{enumerate}

The \emph{($d$-dimensional) edge split operation} replaces an edge $uv$ of graph $G$ with
a new vertex joined to $u$ and $v$, as well as to $d-1$ other vertices of $G$. 

\begin{enumerate}[resume]
  \item The $d$-dimensional edge split operation preserves $\mathcal{R}_d$-independence as well as $d$-rigidity.
\end{enumerate}

We note that properties \textit{(a)-(f)} are shared by all \emph{$1$-extendable abstract rigidity matroids}, see \cite{graver,nguyen_2010}. Since our proofs will only use these properties, our results involving the rigidity matroid remain true for any $1$-extendable abstract rigidity matroid. For simplicity, we only give the statements for the generic rigidity matroid.

\subsection{Unions of rigidity matroids}

We shall also consider unions of rigidity matroids.
Let $\mathcal{M}_i = (E,\mathcal{I}_i), i \in \{1,\ldots,t\}$ be a collection of matroids on a common ground set $E$. The \emph{union} of $\mathcal{M}_1,\ldots,\mathcal{M}_t$ is the matroid $\mathcal{M} = (E,\mathcal{I})$ whose independent sets are defined by
\[\mathcal{I} = \{I_1 \cup \ldots \cup I_t : I_1 \in \mathcal{I}_1, \ldots, I_t \in \mathcal{I}_t\}.\]
Let $r_\mathcal{M}$ and $r_{\mathcal{M}_i}, i \in \{1,\ldots,t\}$ denote the rank functions of the respective matroids.  It is immediate from the definition of $\mathcal{M}$ that $r_\mathcal{M}(E) \leq \sum_{i = 1}^t r_{\mathcal{M}_i}(E)$, with equality if and only if $E$ contains disjoint subsets $E_1,\ldots,E_t$ such that $E_i$ is a base of $\mathcal{M}_i$ for each $i \in \{1,\ldots,t\}$. 

Let $G = (V,E)$ be a graph and let $t$ be a positive integer. We shall denote the $t$-fold union of the $d$-dimensional rigidity matroid of $G$ by $\rigmatroidt(G) = (E,r_d^t)$. Analogously to the $t=1$ case, we define $G$ to be \emph{$\rigmatroidt$-rigid} if $r_d^t(G) = r_d^t(K(V))$, and to be \emph{$\rigmatroidt$-independent} if $r_d^t(G) = |E|$. 
Let us define a pair of vertices $\{u,v\}$ to be \emph{$\rigmatroidt$-linked in $G$} if $r_d^t(G+uv) = r_d^t(G)$. Thus $G$ is $\rigmatroidt$-rigid if and only if every pair of vertices of $G$ is $\rigmatroidt$-linked in $G$.

It follows from property \textit{(c)} above that the addition of a vertex of degree $td$ preserves $\rigmatroidt$-independence. Similarly, it follows from property \textit{(f)} (together with property \textit{(c)}) that the $td$-dimensional edge split operation preserves $\rigmatroidt$-independence.

\begin{lemma}\label{lemma:tdrigid}
    If $n \geq 2td$, then  $r_d^t(K_n) = tdn - t\binom{d+1}{2}$. Hence an $\rigmatroidt$-rigid graph on at least $2td$ vertices contains $t$ edge-disjoint $d$-rigid spanning subgraphs. %Moreover, if $d\geq 2$ or $n \geq 2td + 1$, then $K_n$ is redundantly $\rigmatroidt$-rigid.
\end{lemma}
\begin{proof}
    We show that $K_n$ contains $t$ edge-disjoint $d$-rigid spanning subgraphs $G_1,\ldots,G_t$. Let us choose disjoint subsets of vertices $V_i^j,i \in \{1,\ldots,t\}, j \in \{1,2\}$, each of size $d$, and let $X = V(K_n) - \bigcup_{i=1}^t\bigcup_{j=1}^2 V_i^j$. For each $i \in \{1,\ldots,t\}$, let $G_i$ consist of the complete graph on $V_i^1 \cup V_i^2$, plus the edge sets \[\bigcup_{\ell<i}\left((E(V^1_i,V^1_\ell) \cup E(V^2_i,V^2_\ell)\right) \cup \bigcup_{\ell > i}\left( (E(V^1_i,V^2_\ell) \cup E(V^2_i,V^1_\ell) \right) \cup E(V^1_i,X).\] Now $G_i$ is $d$-rigid, since it has a spanning subgraph obtained from a complete graph by adding vertices of degree $d$, and it is easy to see that $G_i$ and $G_j$ are edge-disjoint for $i \neq j$.
    %All that remains to show is redundant $\rigmatroidt$-rigidity. Let $e = uv$ be an edge of $K_n$. If $d \geq 2$, then $K_{2d}$ is redundantly $d$-rigid. Hence we may construct $t$ edge-disjoint $d$-rigid spanning subgraphs of $K_n-e$ by modifying the above construction to have $u$ and $v$ be contained in $V_1^1$, and constructing $G_1$ from the complete graph on $V_1^1 \cup V_1^2$ minus the edge $e$. Similarly, if $n \geq 2td+1$, then we may choose $V_1^2$ and $X$ so that they contain $u$ and $v$, respectively; with this choice, the construction above will not use $e$. 
\end{proof}

\begin{lemma}\label{lemma:linkedneighbours}
    Let $G = (V,E)$ be a graph and let $v_0 \in V$ be a vertex with $\deg_G(v_0) \geq td+1$. If $r_d^t(G-v_0) = r_d^t(G) - td$, then every pair of vertices $u,v \in N_G(v_0)$ is $\rigmatroidt$-linked in $G-v_0$.     
\end{lemma}
\begin{proof}
    Suppose, for a contradiction, that $\{u,v\}$ is not $\rigmatroidt$-linked in $G-v_0$ for some pair of vertices $u,v \in N_G(v_0)$. This means that $r_d^t(G-v_0+uv) = r_d^t(G-v_0) + 1$. Let $G_0$ be a maximal $\rigmatroidt$-independent subgraph of $G-v_0+uv$; by the previous observation, we must have $uv \in E(G_0)$. We can now obtain an $\mathcal{R}_d^t$-independent subgraph $G'$ of $G$ with $r_d^t(G') = r_d^t(G_0) + td$ by performing a $td$-dimensional edge split on the edge $uv$ in $G_0$. But this means that \[r_d^t(G) \geq r_d^t(G') = r_d^t(G_0) + td = r_d^t(G-v_0) + td + 1 = r_d^t(G) + 1,\] a contradiction.
\end{proof}

\subsection{Tools from probability theory}

We shall need the following standard result from probability theory. We let $\binomialdist{p}$ denote the binomial distribution with parameters $n$ and $p$.

\begin{theorem}(Chernoff bound for the binomial distribution, see, e.g., \cite[Theorem A.1.13]{alon.spencer_2008})  \label{thm:chernoffbound} Let $X\sim \binomialdist{p}$. Then for any $0\leq \eta \leq 1$, \[\Prob(X\leq (1-\eta)np)\;\leq\; e^{-\eta^2np/2}\] %\exp\left(-\frac{\eta^2np}{2}\right).\]
\end{theorem}

We shall also use the following technical lemma.

\begin{lemma}\label{lemma:ordering}
Let $S$ be a finite set of size at least $d$ and fix $s\in S$. Let $\pi $ be a uniformly random ordering of $S$, and let $f(\pi)$ denote the number of elements of $S$ that precede $s$ in $\pi$. Then we have \[\E\big(\min (d, f(\pi))\big)=d - \frac{1}{2}\cdot \frac{d(d+1)}{|S|}.\]
\end{lemma}
\begin{proof} Clearly, $f(\pi)$ is uniformly randomly distributed on $\{0,\ldots,|S|-1\}$. Thus we have
\begin{align*}
  \E\big(\min (d, f(\pi))\big) &= \sum_{i=0}^{|S|-1}\Prob(f(\pi)=i) \cdot \min(d, i) = \sum_{i=0}^{|S|-1}\frac{1}{|S|} \cdot \min(d, i) \\
  &= \frac{1}{|S|}\left(\binom{d+1}{2} + d(|S| - 1 - d) \right) \\
 &= \frac{1}{|S|}\left(d(|S|-1)-\frac{d(d-1)}{2}\right) \\ &= \frac{1}{|S|}\left(d|S|-\frac{d(d+1)}{2}\right) 
\\ &= d - \frac{1}{2}\cdot \frac{d(d+1)}{|S|}.\tag*{\qedhere}
\end{align*}
\end{proof}

\section{Packing rigid spanning subgraphs}\label{section:packing}
In this section we prove \cref{theorem:main1}. 
If $d\leq 2$ or $t=1$, then the statement is true by \cref{theorem:jordan,theorem:soma}. Thus we will assume that $d\geq 3$ and $t\geq 2$. (The proof also works for $d \leq 2$ or $t=1$, but the bound we obtain is weaker.) 

We remark that the factor 10 is rather arbitrary and a more rigorous analysis of the argument presented here would yield a slightly better constant. For a lower bound on vertex-connectivity required in \cref{theorem:main1}, see \cref{sec:concluding}.

The following is the main lemma for our proof of \cref{theorem:main1}.

\begin{lemma}\label{redundvertex}
Let $d,t$ be integers with $d\geq 3$ and $t\geq 2$, and let $G = (V,E)$ be a graph. If the minimum degree of $G$ is at least $(t\cdot 10d(d+1))$, then $G$ has a vertex $v_0\in V$ such that $r_d^t(G-v_0) = r_d^t(G) - td$.
\end{lemma}
\begin{proof}
Our goal is to construct a maximal $\rigmatroidt$-independent subgraph $G'$ of $G$ with minimum degree $td$. As we will see, the existence of such a subgraph quickly implies the lemma.

Fix an orientation $\vec{G}$ of $G$ such that $\delta_{\vec{G}}(v)\geq \lfloor\deg_G(v)/2\rfloor$ for each $v\in V$. It is a well-known result that such an orientation exists (see, e.g., \cite[Theorem 1.3.8]{book}). Then $\delta_{\vec{G}}(v) \geq t\cdot 5d(d+1)$ for each $v\in V$. For a subset $F$ of $E$ let $\vec{F}$ denote the corresponding set of oriented edges. Recall that $N^+_{\vec{F}}(v)$ denotes the set of vertices $u\in V$ for which $vu\in \vec{F}$. %Let $N_{\bar{G}}(v)$ denote the vertices $u$ for which $vu\in E(\bar G)$.

Let $U$ be a random subset of $V$ such that each $v\in V$ is in $U$ independently with probability $1/2$.
% a uniformly random subset $U$ of $V$ of size $\left\lceil |V|/2 \right\rceil+1$. 
  For each $j \in \{1,\dots, t\}$, we recursively define a random subgraph $H_j=(U,F_j)$ of $G[U]$ as follows. 
%For an ordering $\pi$ of the vertices of $U$ let $\overleftarrow{\pi}(v)$ denote the set of the vertices that precede $v$ in $\pi$.
%and for $j=1,\dots, t$ 
Suppose that $H_1,\dots, H_{j-1}$ are already given. Let \[E_j= E(G[U]) - \left(F_1 \cup \ldots \cup F_{j-1}\right).\] Consider a uniformly random ordering $\pi_j$ of the vertices in $U$. 
For each $v\in U$, let 
%$A_j(v)$ denote the set of vertices $u\in N^+_{\vec{E}_j}(v)$ for which $u$ precedes $v$ in $\pi_j$.
$A_j(v)=\big\{u\in N^+_{\vec{E}_j}(v): u \text{ precedes } v \text{ in } \pi_j\big\}$, and fix a subset $B_j(v) \subseteq A_j(v)$ of size $\min(d,|A_j(v)|)$. 
Finally, let
\[F_j=\bigcup_{v\in U}\{vu:{u\in B_j(v)}\}.\]
%in the graph $H_j$ connect $v$ with $\min(d,|A_j(v)|)$ vertices of $A_j(v)$.

Let $F=\bigcup_{j=1}^t F_j$ and  $H=(U,F)$. For each $v\in V-U$, add $v$  to $H$ and connect it with $\min(td,|N^+_{\vec{ G}}(v)\cap U|)$ vertices of $N^+_{\vec{ G}}(v)\cap U$. Let the resulting graph be denoted by $G_0=(V,E_0)$.

\begin{claim}\label{claimindep}
$G_0$ is $\rigmatroidt$-independent.
\end{claim}
\begin{proof}
For each $j\in \{1,\dots t\}$, the graph $H_j$ can be obtained by taking a graph on one vertex and then adding new vertices, one at a time, of degree at most $d$. 
Hence, $H_j$ is $\mathcal{R}_d$-independent. It follows that $H$ is $\mathcal{R}_d^t$-independent. $G_0$ can be obtained from $H$ by adding new vertices of degree at most $td$ and thus $G_0$ is also $\mathcal{R}_d^t$-independent.
\end{proof}

\begin{claim}\label{claimEE0}
$\E(|E_0|)\geq \left(td-\frac{1}{4}\right)n$.
\end{claim}

\begin{proof}
For convenience, we define $n=|V|$ and $k=5d(d+1)$. Let $\eta\leq \frac{1}2$ be a parameter to be chosen later, and let $r=(1-\eta)\frac{k}{2}$.

Let us fix a vertex $v \in V$. Using the Chernoff bound (\cref{thm:chernoffbound}) we obtain that
\begin{equation}\label{eq:smallneighbors}
\Prob\Big(|N^+_{\vec{ G}}(v)\cap U|\leq tr \Big) = \Prob\left(|N^+_{\vec{ G}}(v)\cap U|\leq (1-\eta) \frac{tk}{2} \right) 
\;\leq \; 
%\exp\left(\frac{-\eta^2\frac{1}{2}(tk)}{2} \right)  =
e^{-\eta ^2tk/4 }.
\end{equation}
Let $Q$ denote the event that $v\in U$ and $|N^+_{\vec{ G}}(v)\cap U|>tr$.\\ If $Q$ holds, then $|N^+_{\vec{ E}_j}(v)|\geq tr-td> d$ for every $j\in \{1,\dots,t\}$. \\ Note that $\delta_{\vec{F}_j}(v)=|B_j(v)|=\min(d,|A_j(v)|)$. Thus it follows from 
%the construction of $F_j$ and 
\cref{lemma:ordering} that
\[\E\Big(\delta_{\vec{ F}_j}(v) \;\Big|\; Q \Big) 
%= d-\frac{1}{2}\cdot \frac{d(d+1)}{|N^+_{\vec{E}_j}(v)\cup \{v\}|}
\geq d-\frac{1}{2}\cdot \frac{d(d+1)}{tr-td+1}\geq d-\frac{d(d+1)}{2t(r-d)},\] and hence
\[\E\Big( \delta_{\vec{ F}}(v) \;\Big|\; Q \Big) = \E\Big( \sum_{j=1}^t\delta_{\vec{ F_j}}(v) \;\Big|\; Q \Big) = \sum_{j=1}^t \E\Big( \delta_{\vec{ F_j}}(v) \;\Big|\; Q \Big) \geq td-\frac{d(d+1)}{2(r-d)}.\]
Equation \cref{eq:smallneighbors} and the fact that the two sub-events in the definition of $Q$ are independent together imply $\Prob(Q) \geq \frac{1}{2} \Big(1-e^{-\eta ^2tk/4}\Big)$. Hence,
\[\E\big(\delta_{\vec{ F}}(v)\big)\geq \E\big(\delta_{\vec{F}}(v) \big| Q) \cdot \Prob(Q) \geq \frac{1}{2} \left(1-e^{-\eta ^2tk/4}\right)\left(td-\frac{d(d+1)}{2(r-d)}\right).\]
After summing over the vertices, we get
$$\E(|F|) = \E\bigg(\sum_{v \in V}\delta_{\vec{F}}(v) \bigg) = \sum_{v \in V}\E\big(\delta_{\vec{F}}(v) \big)\geq \frac{1}{2}\left(1-e^{-\eta ^2tk/4}\right)\left( td-\frac{d(d+1)}{2(r-d)}\right)n.$$
Let $D$ denote $E_0-F$. Then $$\E\bigg(\delta_{\vec{D}}(v)\;\bigg|\; v\notin U \;\land\; |N_{\vec{ G}}(u)\cap U|>tr \bigg)= td.$$
Thus by \cref{eq:smallneighbors} we have
 \[\E\big(\delta_{\vec{D}}(v)\big)\geq \frac{1}{2}\left(1-e^{-\eta ^2tk/4}\right)td,\] and by summing over the vertices we obtain
 \[\E(|D|)\geq \frac{1}{2}\left(1-e^{-\eta ^2tk/4}\right) tdn.\]
It follows that 
%$$\E(|D|)\geq \frac{1}{2}\left(1-e^{-4tk}\right) td$$
%and 
%$$\E(|E_0|)=\E(|F|)+E(|D|)\geq $$
\begin{align*}
\E(|E_0|)=\E(|F|)+\E(|D|)
&\geq \left(1-e^{-\eta ^2tk/4}\right)\left( td-\frac{d(d+1)}{4(r-d)}\right)n\\
%&\geq tdn - \frac{1}{5}n-e^{-4tk}tdn
&\geq tdn-\left(   e^{-\eta ^2tk/4}\cdot td  + \frac{d(d+1)}{4(r-d)} \right)n\\
% &\geq tdn-\left(\frac{1}{4}\right)n.
&\geq tdn-\left(   e^{\frac{-\eta ^2tk}{4}+\frac{td}{3}}  + \frac{d(d+1)}{4(r-d)} \right)n,
\end{align*}
where the last inequality follows from the fact that $td\geq 6$ and thus $td<e^{td/3}$.

Let us now fix the value of $\eta$ to be $0.45$. With this choice it is easy to verify that \[\frac{\eta ^2tk}{4} - \frac{td}{3} 
%=\frac{-\eta ^2\cdot 5d(d+1)t}{4} + \frac{td}{3} 
= \left(\frac{\eta^2}{4}-\frac{1}{5(d+1)\cdot 3}\right)t\cdot 5d(d+1)\geq \left(\frac{0.45^2}{4}-\frac{1}{60}\right) 120 =  4.075,\]
where for the inequality we used, again, the assumption that $d \geq 3$ and $t \geq 2$. Note that $d \geq 3$ also implies $d \leq \frac{1}{4}d(d+1)$, and hence
\[r-d=0.55\cdot \frac{1}{2}\cdot 5\cdot d(d+1)-d \geq 1.375\, d(d+1) -0.25\, d(d+1)= 1.125\, d(d+1).\] Thus $$\E(|E_0|)\geq tdn-\left(e^{-4.075}+\frac{1}{4\cdot 1.125}\right)n\geq \left(td-\frac{1}{4}\right)n,$$
which completes the proof of the claim.
\end{proof}

%\end{adjustwidth}
 It follows from \cref{claimEE0} that $$\E\left(|E_0|+\frac{|V-U|}{2}\right) = \E(|E_0|)+\frac{1}{4}n\geq tdn.$$ 
Hence, there exist a set of vertices $U\subseteq V$ and a corresponding spanning subgraph $G_0=(V,E_0)$ such that $$|E_0|+\frac{|V-U|}{2}\geq tdn.$$
Since $E_0$ is $\mathcal{R}_d^t$-independent by \cref{claimindep}, we can extend it to a maximal $\mathcal{R}_d^t$-independent subgraph $G' = (V,E_0 \cup E_1)$ of $G$ by adding a suitable set of edges $E_1 \subseteq E-E_0$.
Then $|E_0 \cup E_1| = r_d^t(G)$, and thus \[|E_1|= r_d^t(G) - |E_0| \leq tdn-t\binom{d+1}{2}-|E_0|\leq \frac{|V-U|}{2}-t\binom{d+1}{2}<\frac{|V-U|}{2},\] 
where in the first inequality we used the fact that $r_d^t(K_n) = tdn - t\binom{d+1}{2}$, which follows from \cref{lemma:tdrigid}.

Hence, there is some vertex $v_0\in V-U$ that is not incident to any edge in $E_1$. 
It follows from the construction of $G_0$ that $d_{G'}(v_0) \leq td$, and thus \[r_d^t(G-v_0) \geq r_d^t(G'-v_0) \geq r_d^t(G) - td.\] On the other hand, since $\deg_G(v_0) \geq td$, and the addition of a vertex of degree $td$ preserves $\rigmatroidt$-independence, we must have $r_d^t(G) \geq r_d^t(G-v_0) + td$. It follows that $r_d^t(G-v_0) = r_d^t(G) - td$.
\end{proof}

\begin{proof}[Proof of \cref{theorem:main1}]
As we noted before, if $d\geq 2$ or $t=1$, then the statement follows from \cref{theorem:jordan,theorem:soma}. Hence it suffices to prove in the case when $d\geq 3, t\geq 2$. 
We prove the statement by induction on the number of vertices. Let $c = t\cdot 10d(d+1)$. If $|V|=c+1$, then $G$ is complete and thus $\mathcal{R}_d^t$-rigid. It follows from \cref{lemma:tdrigid} that $G$ contains $t$ edge-disjoint $d$-rigid spanning subgraphs. 

Suppose now that $|V|>c+1$. By \cref{redundvertex}, there is some vertex $v_0\in V$ such that $r_d^t(G-v_0) = r_d^t(G)-td$.
Let us consider the graph $G' = G-v_0+K(N_G(v_0))$. 
On the one hand, by \cref{lemma:linkedneighbours} and the choice of $v_0$, every pair of vertices $u,v \in N_G(v_0)$ is $\rigmatroidt$-linked in $G-v_0$, and hence $G-v_0$ is $\mathcal{R}_d^t$-rigid if and only if $G'$ is. On the other hand, $G'$ is $c$-connected: indeed, it arises from the $c$-connected graph $G+K(N_G(v_0))$ by deleting a vertex whose neighbor set is a clique, and it is easy to see that deleting such a vertex preserves $c$-connectivity (except for the complete graph on $c+1$ vertices). Hence, by induction, $G'$ is $\mathcal{R}_d^t$-rigid, and thus so is $G-v_0$. 

Since $|V| \geq c+1 \geq 2td+1$, \cref{lemma:tdrigid} implies that $G-v_0$ contains $t$ edge-disjoint $d$-rigid spanning subgraphs. Adding a vertex of degree at least $td$ to $G-v_0$ corresponds to adding a vertex of degree at least $d$ to each of these subgraphs, an operation that preserves $d$-rigidity. Since $\deg_G(v_0) \geq c \geq td$, we conclude that $G$ contains $t$ edge-disjoint $d$-rigid spanning subgraphs, as claimed. 
\end{proof}

\section{Highly connected orientations}\label{section:Thomassen}

In this section we prove \cref{theorem:main2}.
The following ``orientation lemma'' will be a key ingredient in our proof. Given a graph $G = (V,E)$ and a function  $g:V \to \mathbb{Z}_+$, we shall use the notation $g(X) = \sum_{v \in X} g(v)$ for subsets $X \subseteq V$.

\begin{theorem} (Hakimi \cite{hakimi_1965})
\label{hakimi}
Let $G=(V,E)$ be a graph and let $g:V \to \mathbb{Z}_+$ be a function. Then $G$ has an orientation $\DD$ in which
$\rho_\DD(v) = g(v)$ for all $v\in V$ if and only if
\begin{enumerate}
    \item $i_G(X)\leq g(X)$ for all nonempty $X\subseteq V$, and
    \item $|E|=g(V)$ hold. %(where $g(X)=\sum_{v\in X} g(v)$).
\end{enumerate}
\end{theorem}

The same conditions are equivalent to the existence of an orientation in which $g$ specifies the out-degrees.

We shall consider degree-specified orientations of minimally $d$-rigid graphs. 
Given a minimally $d$-rigid graph $G=(V,E)$ with $|V|\geq \binom{d+1}{2}$ and a subset
$R\subseteq V$ with $|R|=\binom{d+1}{2}$, we define the 
in-degree specification function $g_{d,R}$ by putting
$g_{d,R}(v)=d$ for all $v\in V-R$ and $g_{d,R}(r)=d-1$ for all $r\in R$.
We say that an orientation $\DD$ of $G$ is a
{\it $(d,R)$-orientation} if its in-degrees respect the specification $g_{d,R}$.

\begin{lemma}
\label{l1}
Let $G=(V,E)$ be a minimally $d$-rigid graph with $|V|\geq \binom{d+1}{2}$ and let 
$R\subseteq V$ be a set of vertices with $|R|=\binom{d+1}{2}$. Then $G$ has a
$(d,R)$-orientation.
\end{lemma}

\begin{proof}
We have to verify that the conditions of \cref{hakimi} are satisfied.
Let $g=g_{d,R}$.
As $|V|\geq \binom{d+1}{2}\geq d$, we have $|E|=d|V|-\binom{d+1}{2}=g(V)$.
Now consider a set $X\subseteq V$ with $|X|\geq d$.
By \cref{eq:sparse} we have
$$i_G(X)\leq d|X|-\binom{d+1}{2} \leq d|X|- |R\cap X| = g(X),$$
as required. 
Next, consider a set $X\subseteq V$ with $|X|\leq d-1$.
Then,
since $G$ is simple, we have $i_G(X)\leq \binom{|X|}{2}= \frac{|X|(|X|-1)}{2} \leq |X|(d-1)\leq g(X)$, which completes the proof.
\end{proof}

The next lemma provides a lower bound on the number of in-neighbors of certain subsets in $(d,R)$-orientations, 
establishing a link between degree-specified orientations and high vertex-connectivity.

\begin{lemma}
\label{l2}
Let $d$ and $k$ be integers with $k \geq 2$ and $d \geq 4k-4$. Let $G=(V,E)$ be a minimally $d$-rigid graph with $|V|\geq \binom{d+1}{2}$, 
$R \subseteq V$ a set of vertices with $|R|=\binom{d+1}{2}$, and let 
$\DD$ be a $(d,R)$-orientation of $G$. Finally, let $X \subseteq V$ be a set of vertices. If $|X\cap R|\leq \frac{{\binom{d+1}{2}}}2$,
then $X$ has at least $k$ in-neighbors in $\DD$.
\end{lemma}

\begin{proof}
Let $W$ denote the set of in-neighbors of $X$ in $\DD$ and let $R_X=X\cap R$. 
First, assume that $|X|\geq d$. We have $\rho_\DD(v)=d$ for all $v\in X-R_X,$ and $\rho_\DD(v)=d-1$ for all $v\in R_X$. Thus
\[d|X|-|R_X| = \sum_{v\in X}\rho_\DD(v) \leq i_G(X\cup W)\leq d|X\cup W|-\binom{d+1}{2},\] where the last inequality follows from \cref{eq:sparse}. Hence, $d|W|\geq \binom{d+1}{2}-|R_X|\geq \frac{\binom{d+1}{2}}{2}$
and thus $|W|\geq \frac{d+1}{4}> k-1$.

Next, assume that $|X|\leq d-1$. Note that we have
\[ \rho_\DD(X) = \Big(\sum_{v\in X} \rho_\DD(v)\Big) - i_G(X) \geq d|X|-|R_X|- \binom{|X|}{2} = |X| \left(d - \frac{|X|-1}{2}\right)-|R_X|.\]
Since each in-neighbor of $X$ can send at most $|X|$ edges to $X$ in $\DD$, we also have $|X| |W|\geq \rho_\DD(X)$.
It follows that \[|W|\geq d - \frac{|X|-1}{2} -\frac{|R_X|}{|X|} \geq d - \frac{d-2}{2} -1 = \frac{d}{2} \geq 2k - 2\geq k,\] as desired.
\end{proof}

We are now ready to prove Thomassen's conjecture. 
 
\begin{proof}[Proof of \cref{theorem:main2}]
%Let $\Delta=\binom{d+1}{2}$.
Since the $k=1$ case is settled by the theorem of Robbins, we may assume that $k\geq 2$. 
Let $d=4k-4$.
We shall prove that if a graph $G = (V,E)$ has
two edge-disjoint minimally $d$-rigid spanning subgraphs,
then $G$ has a $k$-connected orientation. 
Note that \cref{theorem:main1} guarantees the existence of such subgraphs 
if $G$ is $(320 k^2-560k+240)$-connected. 

Suppose that $G$ has
two edge-disjoint minimally $d$-rigid spanning subgraphs 
$G_1$ and $G_2$. 
%$G_1,G_2,\dots,G_{\binom{\Delta}{2}+2}$
Let us fix a set $R \subseteq V$ of vertices with $|R|=\binom{d+1}{2}$.
We construct the orientation $\DD$ of $G$ by defining the orientations of $G_1$ and $G_2$, and then orienting the remaining edges arbitrarily.
The orientation of $G_1$ is chosen to be a $(d,R)$-orientation.
The orientation of $G_2$ is chosen to be a {\em reversed} $(d,R)$-orientation (in other words, its out-degrees respect
the $(d,R)$-specification).
By \cref{l1}, these orientations exist.

It remains to show that the union $\DD$ of these oriented spanning subgraphs
%, call it $D$,
is $k$-connected.
Suppose not; then there is a subset $S\subseteq V$ with $|S|\leq k-1$ for which $\DD-S$ is not strongly connected.
This means that there is a set $X\subseteq V$ with $V-X-S\not= \varnothing$ and whose in-neighbors are all in $S$ (and hence the
out-neighbors of the set $V-X-S$ are all in $S$). 
If $|X\cap R|\leq \frac{\binom{d+1}{2}}{2}$, then \cref{l2}, applied to the $(d,R)$-orientation of $G_1$,
gives a contradiction. Otherwise, $|(V-X-S)\cap R|\leq \frac{\binom{d+1}{2}}{2}$, and we can apply the
lemma to the orientation of $G_2$ to obtain a contradiction.
\end{proof}

\section{Removable spanning trees}\label{section:Kriesell}

In this section we prove \cref{theorem:Kriesellimproved}. Our proof is an adaptation of the proof of \cref{theorem:soma} from \cite{soma} to the union of the generic $d$-dimensional rigidity matroid and the graphic matroid. The same method can be applied to the $t$-fold union of the rigidity matroid to show that sufficiently highly connected graphs are $\rigmatroidt$-rigid. However, the bound on the required connectivity obtained in this way is quadratic in $t$, in contrast to the linear bound given by \cref{theorem:main1}. 

We recall the following combinatorial lemma from \cite{soma}. For a positive integer $n$, we let $[n] = \{1,\ldots,n\}$, and for a set $X$ and a nonnegative integer $i$, we let $\binom{X}{i}$ denote the family of subsets of $X$ of size $i$.

\begin{lemma}\cite[Lemma 2.5]{soma}\label{lemma:combinatorial}
    Let $n,r,\ell,m$ be nonnegative integers with $\ell+1 \leq m \leq n-1$. Suppose that $H_1,\ldots,H_r$ are distinct proper subsets of $\{1,\ldots,n\}$ with $|H_i \cap H_j| \leq \ell-2$ for every $1 \leq i < j \leq r$. Then
    \[\left| \left\{ S \in \binom{[n]}{m}: \exists j \in \{1,\ldots,r\}, S \subseteq H_j \right\} \right| \leq \binom{n-1}{m}.\]
\end{lemma}

We also recall the central construction of \cite{soma}.
Let us fix $D \geq 2$. Let $G = (V,E)$ be a graph and $\pi = (v_1,\ldots,v_n)$ an ordering of the vertices of $G$. Let us define
\[\Npi(v_i) = \{u \in N_G(v_i): u \text{ precedes } v_i \text{ in } \pi\},\] and let $\degpi(v_i)$ denote $|\Npi(v_i)|$.

We construct a subgraph $G_\pi^D = (V,E^D_\pi)$ of $G$ according to the following rules.

\begin{enumerate}
    \item If $\degpi(v_i) \leq D,$ then in $G_\pi^D$ we connect $v_i$ with every vertex of $\Npi(v_i)$.
    \item If $\degpi(v_i) \geq D+1$ and $\Npi(v_i)$ induces a clique in $G$, then in $G^D_\pi$ we connect $v_i$ with $D$ vertices of $\Npi(v_i)$.
    \item If $\degpi(v_i) \geq D+1$ and $\Npi(v_i)$ does not induce a clique in $G$, then in $G_\pi^D$ we connect $v_i$ with $D+1$ vertices of $\Npi(v_i)$, including two vertices $x$ and $y$ that are not adjacent in $G$.
\end{enumerate}

The following lemma is an adaptation of the main technical lemma in \cite{soma}. %For a set $X$ and a nonnegative integer $i$, we let $\binom{X}{i}$ denote the family of subsets of $X$ of size $i$.

\begin{lemma}\label{lemma:main}%(Adaptation of \cite[Lemma 3.2]{soma})
    Let $G = (V,E)$ be a graph and let $k_0,\ell$ and $D$ be positive integers. Suppose that for each $v \in V$, $\deg_G(v) \geq k_0$, $N_G(v)$ does not induce a clique in $G$, and that if $H_1$ and $H_2$ are the vertex sets of two different maximal cliques of $G[N_G(v)]$, then $|H_1 \cap H_2| \leq \ell-2$. Let $\pi$ be a uniformly random ordering of $V$. Finally, suppose that $k_0,\ell$ and $D$ satisfy the inequality \[k_0^2 + k_0(1-D(D+1)) - \ell(\ell+1) \geq 0.\] Then $\EE(|E^D_\pi|) \geq D|V|$. 
\end{lemma}
\begin{proof}
    We essentially repeat the proof of \cite[Lemma 3.2]{soma}, which is the special case when $\ell = D$ and $k_0 = D(D+1)$. Fix $v \in V$, and let $k$ denote $\deg_G(v)$. 
    %For $0 \leq i \leq k$, we have $\PP(\degpi(v) = i) = \frac{1}{k+1}$. Hence,
    \cref{lemma:ordering} implies that
    \begin{align}\label{eq1}
        \EE(\min(\degpi(v),D)) = D - \frac{1}{2}\cdot \frac{D(D+1)}{k+1}.
    \end{align}  
Let $H_1,\ldots,H_r$ denote the vertex sets of the maximal cliques of $G[N_G(v)]$. For each $i \in \{D+1,\ldots,k\}$, let 
\begin{align*}\mathcal{S}_i &= \left\{ S \in \binom{N_G(v)}{i}: S \text{ induces a clique in } G \right\} \\ &= \left\{ S \in \binom{N_G(v)}{i}: \exists j \in \{1,\ldots,r\}, S \subseteq H_j \right\}.\end{align*}
%Let $A_i$ denote $|\mathcal{S}_i|$. 
Then $|\mathcal{S}_k| = 0$ and, by \cref{lemma:combinatorial}, $|\mathcal{S}_i| \leq \binom{k-1}{i}$ for each $i \in \{\ell + 1, \ldots, k-1\}$. 

Let $Q$ denote the event that $\degpi(v) \geq D+1$ and $\Npi(v)$ does not induce a clique in $G$. If $\degpi(v) = i \geq D+1$, then $Q$ occurs if and only if $\Npi(v) \notin \mathcal{S}_i$. Hence, for $i \geq \ell + 1$ we have
\[\PP(Q|\degpi(v) = i) = 1 - \frac{|\mathcal{S}_i|}{\binom{k}{i}} \geq 1 - \frac{\binom{k-1}{i}}{\binom{k}{i}} = 1 - \frac{k-i}{k} = \frac{i}{k}.\]It follows that
\begin{align}\label{eq2}
\begin{split}
    \PP(Q) &\geq \sum_{i = \ell + 1}^k \PP(\degpi(v)=i)\PP(Q|\degpi(v)=i) 
    \\ &\geq \sum_{i = \ell + 1}^k \frac{1}{k+1}\frac{i}{k} \\ &= \frac{1}{k(k+1)} \cdot \left( \binom{k+1}{2} - \binom{\ell+1}{2}\right) \\ &= \frac{1}{2} - \frac{1}{2} \cdot \frac{\ell(\ell+1)}{k(k+1)}.
\end{split}
\end{align}
If $Q$ does not occur, then $\degGpi(v) = \min(\degpi(v),D).$
\\ If $Q$ occurs, then $\degGpi(v) = D+1 = \min(\degpi(v),D) + 1.$
\\ Hence, by combining \cref{eq1} and \cref{eq2} we obtain
\begin{align}
\begin{split}
    \EE(\degGpi(v)) &= \EE(\min(\degpi(v),D)) + \PP(Q) \\ & \geq D + \frac{1}{2} - \frac{1}{2} \cdot \left(\frac{D(D+1)}{k+1} + \frac{\ell(\ell+1)}{k(k+1)}\right) \\ &\geq D,
\end{split}
\end{align}
where the last inequality follows from the assumption that \[k_0^2 + k_0(1-D(D+1)) - \ell(\ell+1) \geq 0.\] 

    Thus
    \[\EE(|E^D_\pi|) = \EE\left(\sum_{v \in V}\degGpi(v)\right) = \sum_{v \in V}\EE(\degGpi(v)) \geq D|V|.\]
\end{proof}

For the rest of the section we let $\graphicrigid(G)$ denote the union of $\mathcal{R}_d(G)$ and $\mathcal{R}_1(G)$ (i.e., the graphic matroid of $G$), for every graph $G$. Let $r_{\graphicrigid}(G)$ denote the rank of $\graphicrigid(G)$. We define $G = (V,E)$ to be \emph{$\graphicrigid$-independent} if $r_{\graphicrigid}(G) = |E|$,  and \emph{$\graphicrigid$-rigid} if $r_{\graphicrigid}(G) = r_{\graphicrigid}(K(V))$. A pair of vertices $u,v \in V$ is \emph{$\graphicrigid$-linked in $G$} if $r_{\graphicrigid}(G+uv) = r_{\graphicrigid}(G)$. 

Note that a graph is $\graphicrigid$-independent if and only if it can be written as the edge-disjoint union of an $\mathcal{R}_d$-independent graph and a forest. It follows that $\graphicrigid$-independence is preserved by the addition of vertices of degree $d+1$, as well as under the $(d+1)$-dimensional edge split operation.

\begin{lemma}(Adaptation of \cite[Lemma 3.1]{soma})\label{lemma:independent}
    Let $G = (V,E)$ be a graph, and let $\pi = (v_1,\ldots,v_n)$ be an ordering of the vertices of $G$. Suppose that every $\graphicrigid$-linked pair in $G$ is adjacent in $G$. Then $G_\pi^{d+1}$ is $\graphicrigid$-independent.
\end{lemma}
\begin{proof}
    We prove that $F_i = G^{d+1}_\pi[\{v_1,\ldots,v_i\}]$ is $\graphicrigid$-independent by induction on $i$. $F_1$ is a single vertex, which is $\graphicrigid$-independent. Let us thus suppose that $2 \leq i \leq n$. If $F_i$ is constructed from $F_{i-1}$ according to rule \textit{(a)} or \textit{(b)}, then it is obtained from $F_{i-1}$ by the addition of a vertex of degree at most $d+1$, and hence $F_i$ is $\graphicrigid$-independent. Suppose that $F_i$ is constructed from $F_{i-1}$ according to rule \textit{(c)}, and let $x,y$ be vertices as described in the rule. As $x,y$ are nonadjacent in $G$, they are not $\graphicrigid$-linked in $G$, and hence neither in $F_{i-1}$. This means that $F_{i-1}+xy$ is $\graphicrigid$-independent. Since $F_i$ is obtained from $F_{i-1}+xy$ by a $(d+1)$-dimensional edge split, it follows that $F_i$ is also $\graphicrigid$-independent, as claimed.
\end{proof}

We shall also need an analogue of \cref{lemma:tdrigid} for $\graphicrigid$-rigid graphs.

\begin{figure}[t]
    \centering
    \begin{tikzpicture}[x = 1cm, y = 1cm, scale = 1]

       \node (u4) at (-2,0){};
       \node (u1) at (0,1){};
       \node (u2) at (0,0){};
       \node (u3) at (0,-1){};
       \node (v1) at (2,1.8){};
       \node (v2) at (2,1.05){};
       \node (v3) at (2,.325){};
       \node (v4) at (2,-.325){};
       \node (v5) at (2,-1.05){};
       \node (v6) at (2,-1.8){};
       \node (v7) at (4,0){};

       \draw [line width=\normaledge,color=edgeblack] (v1.center) -- (u1.center) -- (u4.center) -- (u2.center) -- (v2.center) -- (u2.center) -- (v3.center);

       \draw [line width=\normaledge,color=edgeblack] (u4.center) -- (u3.center) -- (v4.center) -- (u3.center) -- (v5.center) -- (u3.center) -- (v6.center);

       \draw [line width=\normaledge,color=edgeblack] (u4.center) to[bend left=90, looseness=1.5] (v7.center);

       \draw [fill=vertexblack] (u4) circle (3pt) node[left=3pt]{$u_4$};
       \draw [fill=vertexblack] (u1) circle (3pt) node[above left=0pt]{$u_1$};
       \draw [fill=vertexblack] (u2) circle (3pt) node[above left=0pt]{$u_2$};
       \draw [fill=vertexblack] (u3) circle (3pt) node[below left=0pt]{$u_3$};
       \draw [fill=vertexblack] (v1) circle (3pt) node[right=3pt]{$v_1$};
       \draw [fill=vertexblack] (v2) circle (3pt) node[right=3pt]{$v_2$};
       \draw [fill=vertexblack] (v3) circle (3pt) node[right=3pt]{$v_3$};
       \draw [fill=vertexblack] (v4) circle (3pt) node[right=3pt]{$v_4$};
       \draw [fill=vertexblack] (v5) circle (3pt) node[right=3pt]{$v_5$};
       \draw [fill=vertexblack] (v6) circle (3pt) node[right=3pt]{$v_6$};
       \draw [fill=vertexblack] (v7) circle (3pt) node[right=3pt]{$v_7$};

   \end{tikzpicture}
   \caption{The construction of $T$ in \cref{lemma:graphicrigid} in the case when $\D = 3$ and $d=6$. The complement of $T$ can be obtained from a complete graph $K_{d+1}$ on $d+1$ vertices by successively adding vertices of degree $d$.}
   \label{figure:spanningtree}
\end{figure}

\begin{lemma}\label{lemma:graphicrigid}
    %Let $d \geq 1$ be an integer and 
    Let $\D$ be the smallest integer for which $\binom{\D+1}{2} \geq d$. If $n \geq d + \D + 2$, then $r_{\graphicrigid}(K_n) = (d+1)n - \binom{d+1}{2} - 1$. Hence an $\graphicrigid$-rigid graph on at least $d+\D+2$ vertices contains edge-disjoint spanning subgraphs $G_0$ and $T$ such that $G_0$ is $d$-rigid and $T$ is a tree.
\end{lemma}
\begin{proof}
    We show that $K_n$ contains edge-disjoint spanning subgraphs $G_0$ and $T$ such that $G_0$ is $d$-rigid and $T$ is a tree. We first assume $n = d+\D+2$. Let us label the vertices of $K_n$ as $\{u_1,\ldots,u_{\D+1},v_1,\ldots,v_{d+1}\}$, and let us define the integers $t_0 = 0$, $t_i = \binom{i+1}{2}$ for $i \in \{1,\ldots,\D-1\}$, and $t_{\D} = d$.

    We construct a spanning subgraph $T$ of $K_n$ by adding an edge between $u_i$ and $v_j$ for each $i \in \{1,\ldots,\D\}$ and $j \in \{t_{i-1} + 1,\ldots,t_i\}$, and then adding the edges $v_{d+1}u_{\D+1}$ and $u_iu_{\D+1}$ for each $i \in \{1,\ldots,\D\}$. See \cref{figure:spanningtree}. It is easy to verify that $T$ is a spanning tree of $K_n$. Note that for $i \in \{1,\ldots,\D-1\}$, $u_i$ has exactly $i$ neighbors in $T$ among $v_1,\ldots,v_d$, while $u_\D$ has at most $\D$ such neighbors.

    Let $G_0$ be the complement of $T$ in $K_n$. Then $G_0$ contains the complete graph on $\{v_{1},\ldots,v_{d+1}\}$. Moreover, $u_1$ is adjacent with $\{v_2,\ldots,v_{d+1}\}$ and similarly, for each $i \in \{2,\ldots,\D+1\}$ the vertex $u_i$ has at least $d$ neighbors among $\{v_1,\ldots,v_{d+1},u_1,\ldots,u_{i-1}\}$. It follows that $G_0$ can be constructed from a complete graph by the addition of vertices of degree $d$, and hence it is $d$-rigid, as required.

    The $n > d+\D+2$ case follows from the observation that for such $n$, $K_n$ can be constructed from $K_{d+\D+2}$ by the addition of vertices of degree at least $d+1$.
\end{proof}

\begin{lemma}\label{lemma:graphicrigidgluing}
    Let $\D$ be the smallest integer for which $\binom{\D+1}{2} \geq d$.  If $G_1$ and $G_2$ are complete graphs with $|V(G_1) \cap V(G_2)| \geq d + \D + 2$, then $G_1 \cup G_2$ is $\graphicrigid$-rigid. 
\end{lemma}
\begin{proof}
    Let $G = G_1 \cup G_2$. By \cref{lemma:graphicrigid}, $G_1 \cap G_2$ contains edge-disjoint copies of a spanning tree and $d$-rigid spanning subgraph. We can use these subgraphs and the fact that each vertex of $V(G) - (V(G_1) \cap V(G_2))$ has at least $d+1$ neighbors in $V(G_1) \cap V(G_2)$ to construct edge-disjoint copies of a spanning tree and a $d$-rigid spanning subgraph in $G$.
\end{proof}

\begin{proof}[Proof of \cref{theorem:Kriesellimproved}]Since in the $d \in \{1,2\}$ case we have stronger bounds from the theorem of Nash-Williams \cite{nash-williams_1961} and Tutte \cite{tutte_1961}, and \cref{theorem:jordan}, respectively, we may suppose that $d \geq 3$. (The proof also works for $d \in \{1,2\}$, but the bound we obtain is slightly weaker.) Let $c = d^2+3d+5$. 
    
Suppose, for a contradiction, that $G = (V,E)$ is a $c$-connected graph that is not $\graphicrigid$-rigid. We may assume that $G$ has the least possible number of vertices among all such graphs. We may also assume that $G$ has the largest number of edges among all such graphs on $|V|$ vertices. Then, for each $v \in V, N_G(v)$ does not induce a clique in $G$, for otherwise deleting $v$ would result in a smaller counterexample. (As we noted before, deleting a vertex whose neighbor set is a clique preserves $k$-connectivity unless the graph is a complete graph on $k+1$ vertices.) Furthermore, every $\graphicrigid$-linked pair is adjacent in $G$, for otherwise connecting a nonadjacent $\graphicrigid$-linked pair by an edge would result in a counterexample with more edges. In particular, the $\graphicrigid$-rigid induced subgraphs of $G$ are complete. 

Let $\ell = d + \D + 2$ where \[\D = \left\lceil \sqrt{2d + \frac{1}{4}} -  \frac{1}{2}\right\rceil\]A short calculation shows that $\binom{\D+1}{2} \geq d$. Consider a vertex $v \in V$, and let $H_1,H_2$ be the vertex sets of two different maximal cliques of $G[N_G(v)]$. Then $G[H_1 \cup H_2 \cup \{v\}]$ is non-complete and hence not $\graphicrigid$-rigid. It follows from \cref{lemma:graphicrigidgluing} that \[|(H_1 \cup \{v\}) \cap (H_2 \cup \{v\})| \leq \ell - 1,\] and thus $|H_1 \cap H_2| \leq \ell-2$

Hence we may apply \cref{lemma:main} with $D = d+1, \ell = d + \D + 2$, and $k_0 = d^2 + 3d + 5$. (It is a straightforward, although tedious, calculation to check that these numbers satisfy the condition in the statement of \cref{lemma:main}, under the condition that $d \geq 3$.) It follows that if $\pi$ is a uniformly random ordering of $V$, then $\EE(|E^{d+1}_\pi|) \geq (d+1)|V|$. This implies that there exists some ordering $\pi_0$ of $V$ for which $\EE(|E^{d+1}_{\pi_0}|) \geq (d+1)|V|$. Moreover, by \cref{lemma:independent},  $G_{\pi_0}^{d+1}$ is $\graphicrigid$-independent. But this is impossible, since an $\graphicrigid$-independent graph on at least $d$ vertices can have at most $(d+1)|V|-\binom{d+1}{2} - 1$ edges.

Hence every $c$-connected graph is $\graphicrigid$-rigid. Combining this with \cref{lemma:graphicrigid} (using the observation that $c + 1\geq d+\D+2$), we deduce that every $c$-connected graph contains edge-disjoint copies of a spanning tree and a $d$-rigid spanning subgraph, as required.
\end{proof}

\section{Concluding remarks}\label{sec:concluding}

As we noted in the introduction, we believe that the bound in \cref{theorem:main1} can be replaced by $t \cdot d(d+1)$. The following lemma shows that this would be best possible. For full generality, we state it for arbitrary unions of rigidity matroids.

\begin{lemma}\label{lemma:tightexample}
    Let $d_1,\ldots,d_k$ be a collection of positive integers and define \[K = \left(\sum_{i = 1}^k d_i(d_i+1)\right) - 1.\] There exist infinitely many $K$-connected graphs $G$ that do not contain edge-disjoint spanning subgraphs $G_1,\ldots,G_k$ such that $G_i$ is $d_i$-rigid for each $i \in \{1,\ldots,k\}$. 
\end{lemma}
\begin{proof}
    The proof follows the construction of Lovász and Yemini \cite{lovasz.yemini_1982} for $5$-connected graphs that are not $2$-rigid. Let $G = (V_0,E_0)$ be a $K$-regular $K$-connected graph on $2s$ vertices, where $s$ is a large integer to be determined later. Let $G = (V,E)$ be the graph obtained from $G'$ by splitting every vertex into $K$ vertices of degree one, and then adding a complete graph $G_v$ on the $K$ vertices corresponding to $v$, for each $v \in V_0$. It is not difficult to verify that $G$ is $K$-connected. 
    
    Let $\mathcal{M} = (E,r)$ denote the union of $\mathcal{R}_{d_i}(G)$ for $i \in \{1,\ldots,k\}$. Since we have $E(G) = E_0 \cup \bigcup_{v \in V_0} E(G_v)$, we obtain
    \begin{align*}
        r(E) \leq |E_0| + \sum_{v \in V}r(E(G_v)) &\leq {sK} + 2s \sum_{i = 1}^k \left(  d_iK - \binom{d_i+1}{2} \right) \\ &= 2sK\left( \Big(\sum_{i=1}^k d_i \Big) + \frac{1}{2} - \frac{1}{2} \cdot \frac{K+1}{K} \right) \\ &= |V|\left( \Big(\sum_{i=1}^k d_i \Big) + \frac{1}{2} - \frac{1}{2} \cdot \frac{K+1}{K} \right).
    \end{align*}
    If $s$ is sufficiently large, then the right-hand side is less than $\sum_{i=1}^k \left(d_i|V| - \binom{d_i+1}{2}\right)$, and hence $G$ cannot contain edge-disjoint $d_i$-rigid spanning subgraphs for $i \in \{1,\ldots,k\}$. 
\end{proof}

Finally, we briefly consider the algorithmic aspects of our results. \cref{theorem:main1} is equivalent to the statement that if a graph $G$ is sufficiently highly connected, then there exist $t$ disjoint bases of the generic $d$-dimensional rigidity matroid $\mathcal{R}_d(G)$. It is known that one can construct a random matrix that is a linear representation of the ($t$-fold union of) $\mathcal{R}_d(G)$ with high probability. We can use this fact 
%(and the fact that, given an independence oracle, the matroid partitioning problem can be solved in polynomial time) 
and one of the several polynomial-time algorithms for matroid partition to obtain a randomized algorithm that finds, in expected polynomial time, a packing of $t$ edge-disjoint $d$-rigid spanning subgraphs in graphs satisfying the condition of \cref{theorem:main1}. A similar approach can be used in the case of \cref{theorem:Kriesellimproved} to find a packing of a $d$-rigid spanning subgraph and a spanning tree in suitably highly connected graphs.

We note that our proof of \cref{theorem:main2} is algorithmic in the sense that, given a packing of two $(4k-4)$-rigid spanning subgraphs in a graph, it can be used to explicitly construct a $k$-connected orientation of the graph. Combined with the ideas outlined in the previous paragraph, we obtain a randomized polynomial time algorithm for finding a $k$-connected orientation of a sufficiently highly connected graph. 

\section*{Acknowledgements}
This research was supported by the Hungarian Scientific Research Fund provided by the National Research, Development and Innovation Office, grant Nos. K135421 and PD138102. The second author was supported in part by the MTA-ELTE Momentum Matroid Optimization Research Group and the National Research, Development and Innovation Fund of Hungary, financed under the ELTE TKP 2021‐NKTA‐62 funding scheme. The last author was supported by the Rényi Doctoral Fellowship of the Rényi Institute.

\printbibliography

\end{document}